\documentclass[11pt]{article}

\usepackage[margin=1in]{geometry}

\usepackage{lmodern}

\usepackage{amsmath,amssymb,amsthm,mathtools}

\usepackage{xurl}

\usepackage[colorlinks=true,linkcolor=blue,citecolor=blue,urlcolor=blue]{hyperref}

\hypersetup{pdftitle={Hitting Arithmetic Progressions at the Square-Root Scale},pdfauthor={Samuel Korsky}}

\newtheorem{theorem}{Theorem}[section]

\newtheorem{lemma}[theorem]{Lemma}

\newtheorem{conjecture}[theorem]{Conjecture}

\newcommand{\Z}{\mathbb Z}

\newcommand{\F}{\mathbb F}

\newcommand{\lcm}{\operatorname{lcm}}

\numberwithin{equation}{section}

\title{Hitting Arithmetic Progressions at the Square-Root Scale}

\author{Samuel Korsky}

\date{October 3, 2026}

\begin{document}

\maketitle

\begin{abstract}
\noindent Let $f(N,k)$ be the minimum size of a subset of $\{0,\ldots,N-1\}$ meeting every $k$-term arithmetic progression of positive difference contained in that interval. Improving on a bound of Truss, we prove that
\[
 f(n^2,n)\ge n+n^{1-o(1)}
\]
as $n\to\infty$. Improving on a bound of Xu, for primes $p$ we prove that
\[
 f(p^2,p)\le2p-(\sqrt2-o(1))\sqrt{\frac p{\log p}}.
\]
\end{abstract}

\section{Introduction}

Write $[N]_0=\{0,\ldots,N-1\}$. A subset of $[N]_0$ is a \emph{hitting set} if it meets every $k$-term arithmetic progression contained in $[N]_0$, and $f(N,k)$ is the minimum size of such a set. Throughout, common differences are positive integers. Brown and Freedman~\cite{BrownFreedman} introduced the same function using the translated interval $\{1,\ldots,N\}$.

The complement of a hitting set contains no $k$-term progression. Thus, if $r_k(N)$ is the largest size of a subset of $[N]_0$ containing no such progression, then

\[
 f(N,k)=N-r_k(N).
\]

For fixed $k$, Szemer\'edi's theorem~\cite{Szemeredi} gives $f(N,k)=N-o_k(N)$. Quantitative developments include Gowers's bounds~\cite{Gowers}, the Kelley--Meka and Bloom--Sisask bounds for three-term progressions~\cite{KelleyMeka,BloomSisask}, and the bounds of Leng, Sah, and Sawhney for longer progressions~\cite{LengSahSawhney}. For constructions of progression-free sets, see Behrend~\cite{Behrend}, Rankin~\cite{Rankin}, and O'Bryant~\cite{OBryant}.

Here the progression length grows as a power of the interval length. Brown and Freedman~\cite{BrownFreedman} proved, for fixed $0<\varepsilon<1$, that

\[
 f(N,N^\varepsilon)\ll_\varepsilon N^{1-\varepsilon},
\]

with the progression length rounded to an integer. This matches, up to a constant depending on $\varepsilon$, the elementary lower bound $\lfloor N/k\rfloor$ obtained from that many disjoint blocks of $k$ consecutive integers. For odd primes $p$, they also obtained $f(p^2,p)\le2p-2$.

At the square-root scale $N=n^2$, $k=n$, Truss~\cite{Truss} proved

\[
 n+\tfrac12\sqrt n-2<f(n^2,n)\le q+\left\lceil\frac{(n-1)^2}{q}\right\rceil,
\]

where $q$ is the largest prime at most $n$. The prime number theorem makes the upper bound $2n+o(n)$, while the lower bound gives a square-root excess over the elementary bound $n$. Xu~\cite{Xu} later obtained improved lower bounds for $r_p(p^2)$ through a generalization of Szekeres's construction. By the identity $f(p^2,p)=p^2-r_p(p^2)$, these give improved upper bounds for the same prime-case hitting problem.

We prove the following two bounds.

\begin{theorem}[Lower bound]\label{thm:critical}
For every $\varepsilon>0$, there is an integer $n_0(\varepsilon)$ such that, for all integers $n\ge n_0(\varepsilon)$,
\[
 f(n^2,n)\ge n+n^{1-\varepsilon}.
\]
\end{theorem}

\begin{theorem}[Upper bound for primes]\label{thm:upper}

As $p$ tends to infinity through the primes,

\[
 f(p^2,p)\le2p-(\sqrt2-o(1))\sqrt{\frac p{\log p}}.
\]

\end{theorem}

Theorem~\ref{thm:critical} increases the guaranteed excess above $n$ from square-root size to $n^{1-o(1)}$. Theorem~\ref{thm:upper} gives a saving of order $\sqrt{p/\log p}$ below $2p$ for every sufficiently large prime, improving upon the log-scale savings of Xu.

For the lower bound, suppose a hitting set has $n+s$ points with $s=o(n)$. Most blocks $[jn,(j+1)n)$ then contain exactly one point of the set. Along consecutive such blocks, these points satisfy congruence conditions that force a polynomial description. A prime supplied by Linnik's theorem shows that this polynomial must be linear on sufficiently long runs, so the selected points form an arithmetic progression. We regard these runs as pieces of arithmetic progressions. An elementary count of missing positions lets us join pieces repeatedly while discarding only $o(n)$ selected points. For each fixed $\varepsilon>0$, only a bounded number of joining steps is needed to produce more points on one progression than the hitting-set condition permits.

The upper-bound construction represents each integer $zp+r$ by row $z$ and column $r$ of a $p$ by $p$ array. Most selected points lie in one row and one column; the remaining choices are random points in the first few rows. Every missed progression lies in those rows, where adding points to long empty gaps has small expected cost.

For comparison, above the square-root scale the elementary lower bound is asymptotically sharp: if $k=\lfloor N^\alpha\rfloor$ with $1/2<\alpha<1$, then $f(N,k)\sim N/k$. For the upper bound, let $q$ be the largest prime at most $k$. Then $q\sim k$, while every allowed difference is at most $(N-1)/(k-1)=o(k)$ and hence less than $q$. The first $q$ terms of each progression meet every residue modulo $q$. A residue class containing the fewest points of $[N]_0$ is therefore a hitting set of size $\lfloor N/q\rfloor\sim N/k$.

The square-root scale is therefore where this prime-modulus construction ceases to attain the elementary lower bound automatically. The bounds proved here do not decide whether $f(n^2,n)\ge(1+\eta)n$ for some fixed $\eta>0$ and all sufficiently large $n$. However, we pose the following conjecture:

\begin{conjecture}\label{conj:critical}

As $n\to\infty$,

\[
f(n^2,n)=(2-o(1))n.
\]

\end{conjecture}

Section~\ref{sec:arithmetic} proves the rigidity of long singleton runs, Section~\ref{sec:lower} proves the lower bound, and Section~\ref{sec:upper} gives the construction for primes. All logarithms are natural, and intervals contain their integer points unless stated otherwise.

\section{Long Singleton Runs are Affine}\label{sec:arithmetic}

Let $A\subset[0,n^2)\cap\Z$ meet every contained $n$-term arithmetic progression, and write $|A|=n+s$, where $s\ge0$. We call the points of $A$ selected points. Throughout the lower-bound proof, $n$ is sufficiently large.

Fix $1\le d\le n$. For each $r\in\{0,\ldots,d-1\}$, list the integers $r,r+d,r+2d,\ldots$ below $n^2$ in increasing order. Divide this list into \emph{full blocks} of $n$ consecutive terms, followed by a \emph{terminal block} of fewer than $n$ terms, if any. Each full block is an $n$-term progression, so it must meet $A$; this is why we choose $n$ terms per block.

A full block is a \emph{singleton block} if it contains exactly one selected point, and a \emph{nonsingleton block} if it contains at least two. A \emph{singleton run} is a maximal string of consecutive singleton blocks in one residue class. Its length is its number of blocks, equivalently its number of selected points. Runs for $d=1$ are called \emph{ordinary}; their blocks are $[jn,(j+1)n)$.

Within the residue class $r+d\Z$, we use the coordinate $u$ for the integer $r+du$. A full block thus contains $n$ consecutive $u$-values. We call a run \emph{affine} if its selected coordinates form an arithmetic progression. A difference of $e$ in these coordinates means a difference of $de$ between the original integers.

\begin{samepage}

\begin{lemma}\label{lem:chain}

Across the $d$ residue classes, there are at most $s+d$ nonsingleton blocks and $s+2d$ singleton runs, and at most $2(s+d)$ selected points lie in nonsingleton or terminal blocks. Consequently, for every integer $H\ge1$, at most

\[
 2(s+d)+(H-1)(s+2d)
\]

selected points lie outside runs of length at least $H$. For $d=1$, the three bounds improve to $s$, $s+1$, and $2s$, respectively.

Within a singleton run, every $q$ consecutive selected coordinates contain each residue modulo $q$ exactly once, whenever $q$ does not exceed the run's length. Consecutive selected coordinates differ by at most $n$.

\end{lemma}

\end{samepage}

\begin{proof}

Let $F$ be the total number of full blocks, and let $t_r$ be the number of integers in the terminal block of residue class $r$, taking $t_r=0$ if there is none. The blocks partition all $n^2$ integers, so

\[
 n^2=nF+\sum_{r=0}^{d-1}t_r<nF+dn.
\]

Thus $F>n-d$; the weaker bound $F\ge n-d$ will suffice.

Reserve one selected point in each full block. There are at most

\[
 E:=|A|-F=n+s-F\le s+d
\]

points left over. More explicitly, if $h_B=|A\cap B|$ for a full block $B$ and $\tau$ counts the selected points in terminal blocks, then

\[
 E=\sum_{\text{full blocks }B}(h_B-1)+\tau.
\]

Every nonsingleton block contributes at least one to this sum, so there are at most $E$ such blocks. Also, $h_B\le2(h_B-1)$ when $h_B\ge2$. Hence the total number of points in nonsingleton or terminal blocks is at most $2E\le2(s+d)$.

In each residue class, nonsingleton blocks separate the singleton runs. If there are $k_r$ nonsingleton blocks, there can be at most $k_r+1$ runs. Summing over the $d$ residue classes gives at most $E+d\le s+2d$ runs. A run shorter than $H$ contains at most $H-1$ points, which proves the bound for points outside long runs. When $d=1$, there are exactly $n$ full blocks and no terminal block, so $E=s$ and the sharper bounds follow.

Now take $q$ consecutive singleton blocks in one run. Their union consists of $qn$ consecutive coordinates $u_0,u_0+1,\ldots,u_0+qn-1$, containing exactly $q$ selected coordinates. Partition these $qn$ coordinates into the $q$ progressions

\[
 \{u_0+b+kq:0\le k<n\},\qquad 0\le b<q.
\]

Each corresponds to an $n$-term progression of difference $dq$ in the original integers, so each must contain a selected point. They are disjoint and there are only $q$ selected points altogether, so each contains exactly one. Since the progressions have different residues modulo $q$, the selected coordinates form a complete residue system.

Finally, let $x_j$ and $x_{j+1}$ be the selected coordinates in consecutive singleton blocks. No coordinate strictly between them is selected. If $x_{j+1}-x_j>n$, the $n$ coordinates $x_j+1,\ldots,x_j+n$ would therefore give an $n$-term progression of difference $d$ missed by $A$. This proves the gap bound.
\end{proof}

The residue condition in Lemma~\ref{lem:chain} is strong enough to force every sufficiently long run to be affine. The next proof first interpolates the selected coordinates by a polynomial, then shows that its degree must be one.

\begin{samepage}

\begin{lemma}[Affine rigidity]\label{lem:rigidity}

There is an absolute constant $C_0$ such that every singleton run of length $L\ge T:=\lceil C_0(\log n)^{12}\rceil$ has selected coordinates in an arithmetic progression. Their positive integral difference $e$ satisfies

\begin{equation}\label{prep:affine-difference}
 0\le n-e\le\frac{n-1}{L-1}.
\end{equation}

In a residue class modulo $d$, the corresponding selected integers therefore have common difference $de$.

\end{lemma}

\end{samepage}

\begin{proof}

Index the selected coordinates as $x_0,\ldots,x_{L-1}$. If the first block has index $j_0$, write

\[
 x_i=n(j_0+i)+a_i,\qquad 0\le a_i<n.
\]

Here $a_i$ is the selected point's position within its block.

\emph{A polynomial of small degree.} Write $\Delta x_i=x_{i+1}-x_i$ for the forward difference, and repeat this operation to define $\Delta^k x_i$. Newton interpolation gives

\[
 P(t)=\sum_{k=0}^{L-1}c_k\binom tk,\qquad c_k=\Delta^kx_0\in\Z,\qquad P(i)=x_i\quad(0\le i<L).
\]

Here $\binom t0=1$ and $\binom tk=t(t-1)\cdots(t-k+1)/k!$ for $k\ge1$. The identity $\Delta\binom tk=\binom{t}{k-1}$ for $k\ge1$ explains this interpolation formula: the coefficients are the first entries in the successive rows of differences.

Let $p$ be prime. Whenever $i+p<L$, the overlapping lists $x_i,\ldots,x_{i+p-1}$ and $x_{i+1},\ldots,x_{i+p}$ are both complete residue systems modulo $p$. Their common entries cancel, so $x_{i+p}\equiv x_i\pmod p$. Since $p\mid\binom pj$ for $0<j<p$, it follows that

\[
 \Delta^p x_i=\sum_{j=0}^p(-1)^{p-j}\binom pj x_{i+j}
 \equiv x_{i+p}-x_i\equiv0\pmod p.
\]

Further differences remain divisible by $p$. Thus $p\mid c_k$ whenever $p\le k<L$, and consequently

\[
 \prod_{\substack{p\le k\\p\ {\rm prime}}}p\mid c_k.
\]

For $k\ge2$, taking $k$ differences annihilates the linear term $n(j_0+i)$. Therefore

\[
 |c_k|=|\Delta^ka_0|
 \le\sum_{i=0}^k\binom ki|a_i|
 \le2^kn.
\]

The prime number theorem gives $\log\prod_{p\le k}p\sim k$. Since $\log2+1/5<1$, for sufficiently large $n$ this product exceeds $2^kn$ whenever $k\ge5\log n$. An integer divisible by this product and of absolute value at most $2^kn$ must be zero. Hence $c_k=0$ for all such $k$, and $P$ has degree $D\le5\log n$.

\emph{A prime that rules out higher degree.} Suppose $D\ge2$. We seek a prime $q\le L$ such that $q\equiv1\pmod D$ and $q\nmid c_D$. The coefficient bound implies that $c_D$ has at most $7\log n$ distinct prime divisors: an integer with $k$ such divisors has absolute value at least $2^k$, while $\log|c_D|\le D\log2+\log n$.

Choose a power of two $Q$ with $7\log n<Q\le14\log n$. The $Q$ residue classes

\[
 1+2Dj\pmod{2DQ},\qquad 0\le j<Q,
\]

are distinct and coprime to $2DQ$. Indeed, their representatives are odd and congruent to $1$ modulo $D$, and $Q$ introduces no prime factor other than $2$. Since there are fewer than $Q$ prime divisors of $c_D$, at least one of these classes contains none of them.

We use Linnik's theorem in the following weak form: every residue class coprime to a modulus $m$ contains a prime $q\ll m^6$; see~\cite{Xylouris}. Applied to the chosen class, it gives

\[
 q\ll(DQ)^6\ll(\log n)^{12},\qquad q\equiv1\pmod D,\qquad q\nmid c_D.
\]

The congruence implies $q>D$, since $D\ge2$ and $q$ is prime. Taking $C_0$ sufficiently large ensures $q\le T\le L$.

We may reduce $P$ modulo $q$: each binomial polynomial has denominator $k!$, and these denominators are invertible modulo $q$ because $k\le D<q$. The resulting polynomial over the field $\F_q$ of residues modulo $q$ still has degree $D$, with leading coefficient $b_D=c_D/D!\ne0$. Its values at $0,\ldots,q-1$ are $x_0,\ldots,x_{q-1}$, which contain every residue exactly once by Lemma~\ref{lem:chain}. Thus $P$ permutes $\F_q$.

Set $m=(q-1)/D$, an integer with $1\le m<q-1$. Summing $m$th powers in two ways gives

\[
 0=\sum_{y\in\F_q}y^m
 =\sum_{x\in\F_q}P(x)^m
 =-b_D^m\ne0,
\]

a contradiction. To justify the first and last equalities, the sum of $x^j$ over $\F_q$ is zero for $0\le j<q-1$, whereas the sum of $x^{q-1}$ is $-1$. For $0<j<q-1$, write the nonzero elements as powers of a generator and sum the resulting geometric series. The constant term sums to $q=0$; for $j=q-1$, each nonzero element contributes $1$, giving $q-1=-1$. The polynomial $P^m$ has degree exactly $q-1$ and leading coefficient $b_D^m$, so only that leading term contributes.

Thus $D=1$, since the selected coordinates strictly increase. Their common difference $e$ is a positive integer, and Lemma~\ref{lem:chain} gives $e\le n$. Comparing the first and last offsets yields

\[
 (L-1)(n-e)=a_0-a_{L-1}\le n-1,
\]

which proves~\eqref{prep:affine-difference}.
\end{proof}

\section{The Lower Bound}\label{sec:lower}

We prove Theorem~\ref{thm:critical} by joining affine runs into larger pieces. An elementary count of missing positions will control the points discarded at each step. The proof ends when a surviving piece contains more points on one progression than the hitting-set condition allows.

\begin{proof}[Proof of Theorem~\ref{thm:critical}]
Fix $0<\varepsilon<1/2$; proving the theorem for these values also proves it for larger $\varepsilon$. Put
\[
 Q=\lfloor n^{\varepsilon/3}\rfloor,
 \qquad |A|=n+s,
 \qquad s\le n^{1-\varepsilon}\le n/Q^3,
\]
and suppose for a contradiction that $A$ is a hitting set. For sufficiently large $n$, the threshold $Q$ exceeds $C_0(\log n)^{12}$, so Lemma~\ref{lem:rigidity} applies to every run of length at least $Q$. All limits below keep $\varepsilon$ fixed.

\emph{The initial pieces and their capacity.} Retain only the selected points in ordinary singleton runs of length at least $Q$. By Lemma~\ref{lem:chain}, this discards $O(Qs+Q)=o(n)$ points. Each retained run lies on a progression $e\Z+b$ with
\[
 n-2n/Q\le e\le n.
\]
Moreover, $e$ is odd: two consecutive selected coordinates have opposite parity by Lemma~\ref{lem:chain}. Every residue class modulo $e$ contains an $n$-term progression, since $r+(n-1)e\le ne-1<n^2$ for $0\le r<e$. The other $e-1$ classes therefore require at least $e-1$ points of $A$. This gives the \emph{capacity bound}
\begin{equation}\label{lower:capacity}
 |A\cap(e\Z+b)|\le n+s-e+1\le B:=3n/Q.
\end{equation}

The selected points in each retained run form a \emph{piece}, assigned to that run's progression. Write the pieces as $S_1,\ldots,S_P$ in increasing order. Their containing intervals $[\min S_i,\max S_i]$ in the original integers are disjoint. Joining consecutive pieces on the same assigned progression means taking their union; this preserves disjointness and adds no points. Discarding a piece removes it only from our collection, while $A$ stays fixed.

We will increase the minimum piece size by a factor of $Q$ at each stage. Because $Q$ is a fixed positive power of $n$, only a number of stages bounded in terms of $\varepsilon$ is needed to exceed $B$. Thus it will suffice to keep the total number of discarded points equal to $o(n)$ after each stage.

\emph{Counting missing positions.} Suppose every piece has at least $M\ge Q$ points and $D=o(n)$ selected points have been discarded so far. Write $m_i=|S_i|$, and let $N_i$ count all positions of its assigned progression $e_i\Z+b_i$ between its endpoints, including positions absent from $S_i$. Call these absent positions \emph{holes}, and put $H_i=N_i-m_i$. A hole means a position absent from the retained piece, regardless of whether it belongs to $A$.

There are at most $2n/M$ pieces, since $|A|\le2n$. The containing interval of $S_i$ has length $e_i(N_i-1)$, and these intervals have total length at most $n^2$. Therefore
\[
 \sum_iN_i\le\frac{n^2}{n-2n/Q}+\frac{2n}{M}\le n+\frac{6n}{Q}.
\]
The last inequality uses $M\ge Q$ and $1/(1-2/Q)\le1+4/Q$ for $Q\ge4$. Since $\sum_i m_i=n+s-D$, the total number of holes satisfies
\begin{equation}\label{lower:holes}
 H:=\sum_iH_i\le D+6n/Q=o(n).
\end{equation}

\emph{Growing the pieces.} Assume $Q\le M\le B$. Choose $d$ to be the largest power of two at most $M/Q$, allowing $d=1$. Then
\[
 M/(2Q)<d\le M/Q,
\]
and every assigned difference $e_i$ is coprime to $d$, because it is odd. Call a singleton run modulo $d$ \emph{long} if its length is at least $Q$. Its selected integers have common difference $df$, where $n-2n/Q\le f\le n$. A point of $S_i$ is \emph{matched} if it lies in a long run with $f=e_i$. That entire run then lies on the piece's assigned progression, since it contains that point and has step $de_i$.

There are $O(Q(s+d))$ selected points outside long runs. If a piece meets a long run with $f\ne e_i$, any two common points are separated by at least
\[
 \lcm(e_i,df)=\frac{de_if}{\gcd(e_i,f)}\ge\frac{dnQ}{8}.
\]
Here we used $\gcd(e_i,d)=1$, $\gcd(e_i,f)\le|e_i-f|\le2n/Q$, and $e_i,f\ge n/2$.

Fix a residue modulo $d$, and compare the piece intervals with the long-run intervals in that residue. Each family consists of disjoint intervals. Every nonempty intersection begins at the left endpoint of one of these intervals, and disjointness allows that endpoint to be counted only once. There are $P$ piece intervals for each residue and at most $s+2d$ long runs altogether. Across all residues, there are therefore at most $dP+s+2d$ intersections. Their total length, measured in the original integers, is at most $dn^2$, since the intersections are disjoint within each residue. An intersection of length $\ell$ with $f\ne e_i$ contains at most $1+8\ell/(dnQ)$ common points. Thus the total number $U$ of unmatched retained points satisfies
\begin{equation}\label{lower:unmatched}
 U\ll Q(s+d)+dP+n/Q\ll n/Q=o(n).
\end{equation}
For the last bound, use $Qd\le M\le B$, $dP\le2n/Q$, and $Qs\le n/Q^2$.

Let $u_i$ count the unmatched points of $S_i$. Discard pieces with $H_i+u_i\ge m_i/8$. Summing $m_i\le8(H_i+u_i)$ over these pieces bounds their total size by $8(H+U)=o(n)$.

Every surviving piece has matches in at least $3d/4$ residues. To see this, each $d$ consecutive positions of its assigned progression cover every residue modulo $d$, because $\gcd(e_i,d)=1$. Between the piece's endpoints, each residue therefore has at least
\[
 \lfloor N_i/d\rfloor\ge m_i/d-1\ge m_i/(2d)
\]
positions, using $m_i/d\ge Q\ge2$. In a residue with no match, all these positions are holes or unmatched points. If $d/4$ residues had no match, they would account for at least $m_i/8$ such positions, contrary to the deletion rule.

Choose one matched representative per available residue in each surviving piece, consistently for both of its neighbors. Adjacent pieces share at least $d/2$ residues. If their assigned progressions differ, their representatives lie in different singleton runs, with a nonsingleton block between them. Within each residue, the intervals between representatives of adjacent pieces have disjoint interiors, so each separating block is counted at most once. There are at most $s+d$ nonsingleton blocks. Hence at most $2s/d+2$ adjacent pairs have different assigned progressions.

Join maximal consecutive groups on the same progression. There are at most $2s/d+3$ groups. Discard those with fewer than $QM$ points. A discarded group has at most $QM$ points, and every group has at most $B$ points by~\eqref{lower:capacity}. The cost of this deletion is therefore at most
\[
 \begin{aligned}
 (2s/d+3)\min\{QM,B\}
 &\le (2s/d)QM+3B\\
 &\le4Q^2s+3B\le13n/Q=o(n).
 \end{aligned}
\]
We used $1/d<2Q/M$ and $s\le n/Q^3$. Every surviving piece now has at least $QM$ points, and the total number of points discarded so far is still $o(n)$.

Starting with $M=Q$, repeat this step until the minimum size exceeds $B$. Since $\log Q\sim(\varepsilon/3)\log n$, only $O_\varepsilon(1)$ stages occur. At each stage, \eqref{lower:holes} and~\eqref{lower:unmatched} keep the additional loss equal to $o(n)$; the number of stages is bounded, so the total loss remains $o(n)$. Some piece therefore survives with more than $B$ points, contradicting the capacity bound. This proves the theorem.
\end{proof}

\section{The Upper Bound for Primes}\label{sec:upper}

\begin{proof}[Proof of Theorem~\ref{thm:upper}]

Write each integer in $[0,p^2)$ uniquely as $zp+r$, with $0\le z,r<p$, and view it as the point in row $z$ and column $r$ of a $p$ by $p$ array. Choose integers $3\le h<H$ with $(h+1)H<p$. In each of columns $1,\ldots,H$, select one point in rows $0,\ldots,h-1$: fix $p+1$ in column $1$, and choose the other rows independently and uniformly. Let $F$ be these $H$ points. Add column $0$ below these rows and the part of row $h-1$ to the right of column $H$, giving

\[
 A_0=\{zp:h\le z<p\}\cup F\cup\{(h-1)p+r:H<r<p\}.
\]

Then $|A_0|=2p-h-1$. We first show that every $p$-term progression missed by $A_0$ lies in the initial strip $[0,hp)$, consisting of rows $0,\ldots,h-1$.

Every column contains a selected point, so $A_0$ meets every progression of difference $p$. A contained $p$-term progression has difference at most $(p^2-1)/(p-1)=p+1$. The only one of difference $p+1$ starts at $0$ and contains the selected point $p+1$. If a missed progression has difference $d<p$, its $p$ terms cover every residue modulo $p$, so it contains a unique multiple $zp$ of $p$, necessarily with $z<h$. If the progression crosses $hp$, its last point below $hp$ is $x=(h-1)p+r$, where $p-d\le r\le H$: crossing gives the lower bound, and the selected part of row $h-1$ gives the upper bound. Put $u=p-d$. For some $k\ge1$,

\[
 x-zp=k(p-u)<hp.
\]

Since $u\le H$ and $(h+1)(p-H)>hp$, we have $k\le h$. But reduction modulo $p$ gives $r+ku\equiv0\pmod p$, contradicting

\[
 0<r+ku\le(h+1)H<p.
\]

Thus every missed progression lies in $[0,hp)$. Its difference is at most $h$, since its endpoint separation is $(p-1)d<hp$ and $h<p$.

We repair the remaining progressions one difference at a time. Fix $1\le d\le h$. List the positions in each residue class modulo $d$ within the strip in increasing order. The points of $F$ divide each list into \emph{gaps}: consecutive positions containing no point of $F$, including the gaps at the two ends. Gap length counts positions in this list, whose successive integers are $d$ apart. A list with $v$ points of $F$ has $v+1$ gaps, counting empty ones, so there are $H+d$ gaps in total. In a gap of length $L$, add the positions numbered $p,2p,3p,\ldots$ from its start. This uses $\lfloor L/p\rfloor$ points and leaves no $p$ consecutive positions unselected, hence no missed progression of difference $d$.

Label each noninitial gap by the point of $F$ immediately before it, and condition on that point's row. Any $kp$ consecutive positions following it in the residue list contain exactly $k$ integers in each column, because $d<p$ and $p$ is prime. If they all lie in the strip, their rows within each column are distinct. The event $L\ge kp$ requires every random choice to avoid those rows. Ignore column $1$, whose point is fixed, and the column on which we conditioned. In each of the at least $H-2$ remaining columns, the independently chosen row avoids the $k$ forbidden rows with probability $1-k/h$. Thus

\[
 \Pr(L\ge kp)\le(1-k/h)^{H-2}\le e^{-k(H-2)/h}\qquad(1\le k\le h).
\]

The bound is uniform in the conditioned row and therefore also holds without conditioning. For an initial gap, use the first $kp$ positions of its residue list; the same bound applies. If these positions leave the strip, or if $k>h$, the event is impossible. Using $\mathbb E\left[\lfloor L/p\rfloor\right]=\sum_{k\ge1}\Pr(L\ge kp)$ and summing over all gaps and differences, the expected total repair cost is at most

\[
 \sum_{d=1}^h\frac{H+d}{e^{(H-2)/h}-1}
 =\frac{hH+h(h+1)/2}{e^{(H-2)/h}-1}.
\]

Taking the union of the added points repairs every remaining progression. Some choice of $F$ therefore gives a hitting set of size at most $2p-h-1$ plus this expected cost.

Finally choose

\[
 h=\left\lfloor\sqrt{\frac{2p}{\log p+4\log\log p}}\right\rfloor,
 \qquad H=\left\lfloor\frac{p-1}{h+1}\right\rfloor.
\]

These parameters satisfy the required inequalities for large $p$, and

\[
 \frac{H-2}{h}=\frac{\log p+4\log\log p}{2}+o(1).
\]

The expected repair cost is consequently $O\left(\sqrt p/(\log p)^2\right)=o(h)$. Since $h=\left(\sqrt2-o(1)\right)\sqrt{p/\log p}$, the theorem follows.
\end{proof}

\section*{Acknowledgements}

The v1 of this paper had a lower bound of $n + \Omega(n^{1/2})$ due to the author; GPT Astra was used to improve the bound to $n + \Omega(n^{3/4 - o(1)})$. The author used ideas from that draft to establish Lemma \ref{lem:rigidity} and the bound in this work. GPT Astra then continued to refine the proof and claimed a lower bound of $n + \Omega\left(n\left(\frac{\log\log{n}}{\log{n}}\right)^2\right)$, but the argument is much more technical and so was not included here. 

On the upper bound side, the author originally provided an alteration argument that reached a bound of $2p - O\left(p^{1/3 - o(1)}\right)$. GPT Astra sharpened the bound to its current form, and also was able to improve the savings to $\sqrt{q}(\log{q})^c$ for some $c > 0$ along an infinite sequence of primes $q$ with bounded prime gaps. Again, that argument was quite technical and is excluded here. The author takes full responsibility for the correctness of the results.

\end{document}